\documentclass[twoside]{article}
\usepackage{enumerate,mathrsfs,amsfonts,amsxtra,latexsym,amssymb,amsthm,verbatim,amsmath,graphicx}
\usepackage{mathpazo}
\usepackage[dvipdfm, colorlinks, linkcolor=green, anchorcolor=blue, citecolor=red]{hyperref}

\catcode`@=12

\newtheorem{theorem}{Theorem}[section]
\newtheorem{definition}[theorem]{Definition}
\newtheorem{lemma}[theorem]{Lemma}

\newtheorem{proposition}[theorem]{Proposition}

\begin{document}
\abovedisplayskip=6pt plus 1pt minus 1pt \belowdisplayskip=6pt
plus 1pt minus 1pt
\thispagestyle{empty} \vspace*{-1.0truecm} \noindent
\vskip 10mm

\begin{center}{\large Kauffman-Harary conjecture for alternating virtual knots} \end{center}

\vskip 5mm
\begin{center}{Zhiyun Cheng\\
{\small School of Mathematical Sciences, Beijing Normal University
\\Laboratory of Mathematics and Complex Systems, Ministry of
Education, Beijing 100875, China
\\(email: czy@bnu.edu.cn)}}\end{center}

\vskip 1mm

\noindent{\small {\small\bf Abstract} In 1999, Kauffman-Harary conjectured that every non-trivial Fox $p$-coloring of a reduced, alternating knot diagram with prime determinant $p$ is heterogeneous. Ten years later this conjecture was proved by W. Mattman and P. Solis. Mathew Williamson generalized this conjecture to alternating virtual knots and proved it for certain families of virtual knots. In the present note, we use the methods of W. Mattman and P. Solis to give an affirmative answer to the Kauffman-Harary conjecture for alternating virtual knots.
\ \

\vspace{1mm}\baselineskip 12pt

\noindent{\small\bf Keywords} alternating virtual knot; Kauffman-Harary conjecture\ \

\noindent{\small\bf MR(2010) Subject Classification} 57M25\ \ {\rm }}

\vskip 1mm

\vspace{1mm}\baselineskip 12pt

\section{Introduction}
The fundamental problem in classical knot theory is the classification of knots up to ambient isotopy. One of the simplest knot invariant is the 3-coloring invariant, or more generally, the Fox $n$-coloring \cite{Fox1961}. A \emph{Fox $n$-coloring} denotes a representation from the fundamental group of the knot complement to the dihedral group. Or equivalently speaking, a Fox $n$-coloring is an assignment of $\{0, 1, \cdots, n-1\}$ to the arcs of a knot diagram such that at each crossing the average of the integers at the under-arcs equals to the integer assigned to the over-arc mod $n$. Usually we denote col$_n(K)$ the number of distinct Fox $n$-colorings of a knot $K$. It is not difficult to check that this is a knot invariant. From this invariant one can define some other knot invariants, such as the minimum number of colors \cite{Kau2008}. The minimum number of colors, which is denoted by min col$_n(K)$, is the minimum number of distinct colors that are needed to produce a non-trivial Fox $n$-coloring among all diagrams of $K$. Evidently, the minimum number of colors min col$_n(K)$ is an invariant of the knot $K$. However in general it is very difficult to compute. See \cite{Kau2008,Kau2012} and \cite{Sai2010} for some recent progress.

In 1999, L. Kauffman and F. Harary \cite{Har1999} conjectured that the minimum number of colors min col$_p(D)$ of a reduced alternating knot diagram $D$ with prime determinate $p$ is exactly the crossing number of $D$, hence it equals to the crossing number of the knot $K$ that $D$ represents. It means that for any non-trivial Fox $p$-coloring of $D$, different arcs are assigned by different colors, i.e. min col$_p(D)=c(K)$. According to \cite{Mat2009} we say a coloring is \emph{heterogeneous} if it assigns different colors to different arcs. We remark that in general min col$_p(D)\neq$ min col$_p(K)$. For instance min col$_5(T_{2, 5})=4$, but $c(T_{2, 5})=5$ \cite{Kau2008}. Kauffman-Harary conjecture has been proved to be valid for several cases, see \cite{Asa2004,Dow2010}, and finally it was proved by T. W. Mattman and P. Solis in \cite{Mat2009}. There are several versions of generalized Kauffman-Harary conjecture. For example, in \cite{Asa2004} it was conjectured that if $K$ is an reduced alternating diagram of a prime knot then different arcs represent different elements of $H_1(M_K^{(2)}, Z)$, here $M_K^{(2)}$ denotes the 2-fold cover of $S^3$ branched over $K$. If the determinate of $K$ is a prime integer then this generalized conjecture is equivalent to the original version. On the other hand it is natural to consider another version of generalized Kauffman-Harary conjecture that replaces the classical alternating knots by alternating virtual knots. In \cite{Mat2007}, by studying the $k$-swap move Mathew Williamson showed that this generalized Kauffman-Harary conjecture holds for some alternating virtual pretzel knot diagrams and alternating virtual 2-bridge knot diagrams. The main aim of this note is to prove the theorem below.
\begin{theorem}
Let $D$ be a reduced alternating virtual knot diagram with a prime determinate $p$, then each non-trivial Fox $p$-coloring of $D$ is heterogeneous.
\end{theorem}

In section 2, we review the definition of quandle and the basic background of virtual knot theory. Some classical invariants, such as the coloring invariant and the validity of the determinate of virtual knots will be discussed. In section 3, based on the methods of Mattman and Solis developed in \cite{Mat2009} we will prove the generalized Kauffman-Harary conjecture as shown in Theorem 1.1.

\section{Dihedral quandle and virtual knots}
In this section, we review some basic results on Fox $n$-coloring and virtual knot theory. First we take a quick review of the definition of quandle \cite{Joy1982,Mat1984}.
\begin{definition}
A quandle $Q$ is a nonempty set with a binary operation $\ast$: $Q\times Q\rightarrow Q$ which satisfies the following conditions:
\begin{enumerate}
\item $a\ast a=a$ for any $a\in Q$
\item $x\ast a=b$ have the only solution $x\in Q$, for any $a, b\in Q$
\item $(a\ast b)\ast c=(a\ast c)\ast (b\ast c)$ for any $a, b, c\in Q$
\end{enumerate}
\end{definition}

We can use the elements of a (finite) quandle to color the arcs of a knot diagram, such that at each crossing the following condition holds. See Figure 1 below.
\begin{center}
\includegraphics{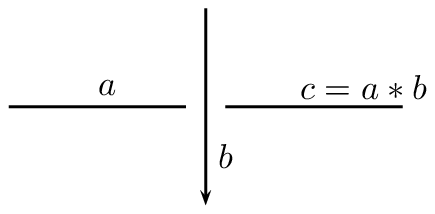} \centerline{\small Figure
1\quad}
\end{center}
We name such coloring a \emph{proper coloring}. Notice that the three conditions of the binary operation $\ast$ exactly correspond to the three Reidemeister moves, it is not difficult to find that with a given finite quandle the number of proper colorings is a knot invariant. In particular, let $D_n=\{0, 1, \cdots, n-1\}$ and the binary operation $\ast$ is defined by $a\ast b=2b-a$ mod $n$, obviously $(D_n; \ast)$ is a quandle and we name it the \emph{dihedral quandle} of order $n$. In this case each proper coloring is exactly a Fox $n$-coloring. Note that the trivial coloring is allowed here, by a trivial coloring we mean all the arcs are assigned the same color.

In order to investigate Fox $n$-colorings, it is convenient to introduce the \emph{coloring matrix} of a knot diagram. Given a knot diagram $D$, denote the crossings and arcs of $D$ by $\{c_1, \cdots, c_k\}$ and $\{a_1, \cdots, a_k\}$ respectively. The $k\times k$ coloring matrix $M(D)$ of $D$ can be defined as below
\begin{center}
$m_{ij}(D)=\left\{
             \begin{array}{ll}
               2, & \hbox{if $a_j$ is the over-arc at $c_i$;} \\
               -1, & \hbox{if $a_j$ is an under-arc at $c_i$;} \\
               0, & \hbox{otherwise.}
             \end{array}
           \right.$
\end{center}
In particular if $a_j$ is the over-arc and an under-arc at $c_i$, then $m_{ij}(D)=2-1=1$. Notice that when $D$ is a reduced and alternating, the entries of the each row and each column of $M(D)$ consist of 2, -1, -1, and $(k-3)$ 0's. With the coloring matrix $M(D)$, for each proper coloring one can represent the assigned integers on arcs of $D$ by a vector $X$ that satisfies
\begin{center}
$M(D)X=0$ mod $n$.
\end{center}
Choose any $(k-1)\times (k-1)$ minor of $M(D)$, for example the $(k, k)$ minor $M'(D)$, the \emph{determinate} of the knot $K$ that $D$ represents can be defined as the absolute value of the determinate of $M'(D)$. It follows obviously that $D$ has a non-trivial Fox $n$-coloring if and only if
\begin{center}
gcd$(n$, det $K)>1.$
\end{center}

Now let us turn to the virtual knot theory which was first proposed by L. Kauffman in \cite{Kau1999}. The motivation of introducing virtual knots mainly comes from the interpretation of virtual links as stably embedded circles in thickened surfaces. On the other hand virtual knots can also be used to realize all Gauss diagrams. With the viewpoint of knot diagram, virtual knot diagrams contain an extra type of crossing besides the classical crossing, say virtual crossing. A virtual crossing is usually denoted by a small circle placed around the crossing point. We say a pair of virtual knot diagrams represent the same virtual knot if one can be obtained from the other by a sequence of generalized Reidemeister moves. See Figure 2.
\begin{center}
\includegraphics{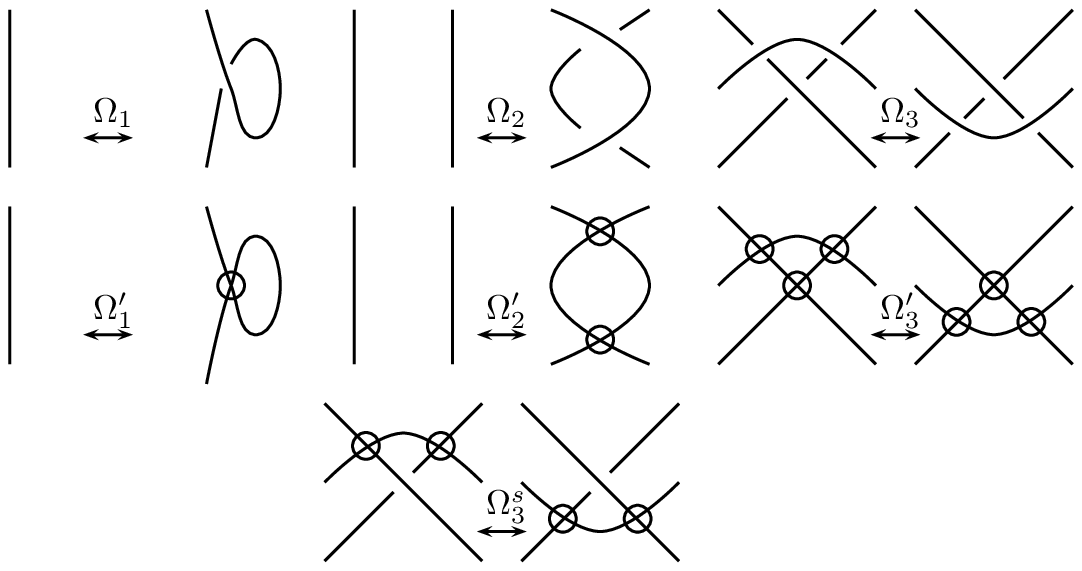} \centerline{\small Figure
2\quad}
\end{center}
We say a virtual knot diagram is \emph{alternating} if one travels along the knot diagram, the over-crossing and under-crossing succeed each other in alternation. Here the virtual crossings are not considered. In particular when the diagram has no virtual crossing, this definition coincides with the classical case.

Given a dihedral quandle $D_n$ and a virtual knot diagram $D$, similar to the classical case one can assign the elements of $D_n$ to the arcs of $D$ such that at each classical crossing the relation in Figure 1 holds. Here by an \emph{arc} of a virtual knot diagram we mean a strand from an under-crossing to another under-crossing, with only over-crossings and virtual crossings in its interior. We say such a coloring is a \emph{proper coloring} on virtual knot diagram and it is evident that the number of proper coloring is also an invariant for virtual knots. Before defining the coloring matrix of a virtual knot diagram, it would be benefit to make a little modification to the notion of reduced virtual knot diagram. Recall that in classical knot theory by a reduced knot diagram we mean that the diagram contains no nugatory crossing point, i.e. there does not exist a simple circle in the projection plane meeting the knot diagram at a single crossing. Given a virtual knot diagram $D$, we say a $($classical or virtual$)$ crossing is \emph{nugatory} if there exists a simple circle meeting $D$ at that crossing, after some virtual moves $\{\Omega_1', \Omega_2', \Omega_3', \Omega_3^s\}$. The following figure gives an example of nugatory classical crossing in a virtual knot diagram.
\begin{center}
\includegraphics{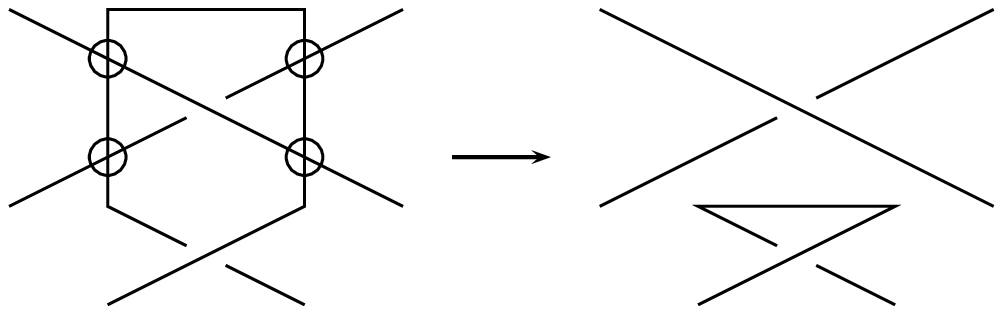} \centerline{\small Figure
3\quad}
\end{center}
Now if a virtual knot diagram contains no nugatory crossing, we say this diagram is \emph{reduced}. Given a virtual knot diagram the coloring matrix can be defined similarly as the classical case. Similar to the classical case, if a virtual knot diagram is reduced and alternating then each row and each column of the coloring matrix consists of 2, -1, -1 and several 0's. We remark that the modification of the original definition of reduced diagram is necessary, since it is easy to observe that the left virtual knot diagram in Figure 3 has no heterogeneous coloring.

It is natural to define the determinate of a virtual knot diagram $D$ with classical crossing number $k$ to be the absolute value of the determinate of any $(k-1)\times (k-1)$ minor of the coloring matrix. However this definition is ambiguous, in fact the result depends on the choice of the row and column. For example, let us consider the virtual knot diagram given in Figure 4.
\begin{center}
\includegraphics{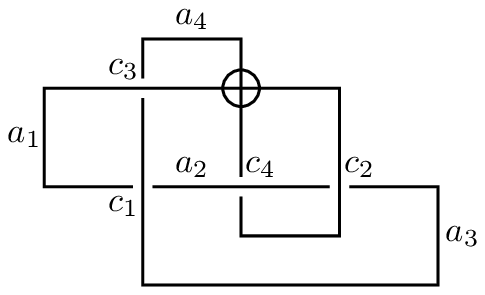} \centerline{\small Figure
4\quad}
\end{center}
Denote the arcs and classical crossings of the diagram above by $a_i$ and $c_i$ $(1\leq i\leq4)$ respectively. Then the associated coloring matrix can be represented by
\begin{center}
$\begin{pmatrix}
    -1 & -1 & 2 & 0 \\
    2 & -1 & -1 & 0 \\
    2 & 0 & -1 & -1 \\
    -1 & 2 & 0 & -1 \\
\end{pmatrix}$.
\end{center}
Direct calculation shows that the absolute value of the $(1,1)$ minor equals 1, but that of the $(4,4)$ minor equals 3. Hence different choice of minors will give different values. It follows that the determinate of a virtual knot diagram can not be defined simply in this way. Usually the \emph{determinate} of a virtual knot diagram is defined to be the greatest common divisor of the absolute values of all $(k-1)\times(k-1)$ minors \cite{Saw1999,Mat2007}. Later we will show that the original definition is valid for reduced alternating virtual knot diagrams. Before doing this, we want to spend a little time on investigating the cause of the strange phenomenon above.

Choose a labeling of the arcs and crossings suitably, it is possible to make the entries $m_{11}, \cdots, m_{kk}$, $m_{12}, \cdots, m_{k-1,k}$ and $m_{k1}$ are all -1. For example the diagram in Figure 4 is labeled in this way. The rest positions are all 0's except $k$ 2's. As before here $k$ denotes the number of classical crossings and for simplicity the diagram is assumed to be reduced. If the diagram is classical then the positions of these $k$ 2's are not arbitrary, since if this coloring matrix can be realized as a classical knot diagram then there are some conditions for the positions of these $k$ 2's. In this case, multiplying certain rows of the coloring matrix by -1 results a matrix with its rows adding to 0 \cite{Liv1993}. With some linear algebra it is not difficult to prove that the absolute value of the determinate of all $(k-1)\times(k-1)$ minors are equal. However in the virtual knot case, the positions of these $k$ 2's are arbitrary. Hence the coloring matrix of some virtual knot diagram can not be converted to a matrix by multiplying some rows by -1, such that the sum of the rows of this new matrix involves only 0 entries. This explains the reason why the original definition of determinate is not valid for virtual knot diagram.

Although the original definition of the determinate depends on the choice of the minor, it is still well-defined for reduced alternating virtual knot diagram.
\begin{lemma}
Let $D$ be a reduced alternating virtual knot diagram with $k$ classical crossings, and $M(D)$ the coloring matrix of $D$, then the determinates of all $(k-1)\times(k-1)$ minors are equal.
\end{lemma}
\begin{proof}
The proof is analogous to the classical case. Since $D$ is reduced and alternating, it follows that each row or column of $M(D)$ consists of 2, -1, -1 and $(k-3)$ 0's. Therefore the sum of rows and the sum of columns are both 0. With some linear algebra it is easy to find that
\begin{center}
$M_{ij}=\frac{1}{k^2}$det$(M(D)+N)$,
\end{center}
here $M_{ij}$ is the $(i,j)$ minor of $M(D)$ and $N$ denotes the $k\times k$ matrix with all entries 1. The proof is complete.
\end{proof}

Let $D$ denote a reduced alternating virtual diagram with $k$ classical crossings, according to Lemma 2.2 the determinate of $D$ equals the absolute value of the determinate of any $(k-1)\times (k-1)$ minor of $M(D)$.

Before ending this section we give two properties about the determinate of a reduced virtual alternating knot diagram, which are well-known in the classical case.
\begin{proposition}
Let $D$ be a reduced alternating virtual knot diagram with $k$ classical crossings, then
\begin{enumerate}
  \item det $D\geq k$.
  \item In additional, if $D$ is the connected sum of two reduced alternating virtual knot diagrams, say $D_1$ and $D_2$, then det $D$=det $D_1 \times$ det $D_2$.
\end{enumerate}
\end{proposition}
\begin{proof}
The proof is divided into two parts.
\begin{enumerate}
  \item The basic idea of the proof mainly comes from \cite{Bal2001}, here we sketch the outline. Consider the shadow of the diagram $D$, which can be regarded as a 4-valent graph $G$ by placing vertices at each classical crossings. If $D$ contains no virtual crossing then this graph is a planar graph. In general $G$ may be not planar if $D$ contains some virtual crossings. As before we use $a_i$ and $c_i$ to denote the arcs and crossings of $D$, such that $a_i$ goes over $c_i$. If $a_i$ is an under-arc at $c_j$ then we orient the edge $v_iv_j$ from $v_j$ to $v_i$, here $v_i$ is the vertex corresponding to $c_i$. In this way we obtain a directed graph with in-degree 2 and out-degree 2. We still use $G$ to name it. The key point is that $G$ contains no articulation vertex if the diagram $D$ is reduced. In order to see this, it suffices to observe that if there is an articulation vertex in $G$ then the Gauss diagram of $D$ contains an isolated chord which has no intersection with other chords. Hence by applying some virtual moves $D$ can be deformed into a new diagram such that the articulation vertex corresponds to a nugatory crossing. According to \cite{Bal2001}, the determinate of $D$ equals the number of Euler circuits of $G$. With the aid of interlace polynomial \cite{Arr2000}, it was proved that if $G$ is a connected in-degree 2 out-degree 2 directed graph with $k$ vertices, and it contains no articulation vertices, then the number of Euler circuits of $G$ is at least $k$. The proof is finished.
  \item According to Lemma 2.2, the determinate of all $(k-1)\times(k-1)$ minors are equal. This ensures the validity of the proof for classical knots. Hence we omit it here.
\end{enumerate}
\end{proof}

\section{The proof of Theorem 1.1}
In this section we mimic the approach of Mattman and Solis to give the proof of Theorem 1.1.
\begin{proof}
Assume the number of classical crossing points of $D$ is $n$, as before we can construct an $n\times n$ coloring matrix $M(D)$. According to Lemma 2.2 the absolute values of the determinates of all $(n-1)\times(n-1)$ minors are the same. Without loss of generality, let us consider the submatrix $M'(D)$ formed by deleting the $n$-th row and $n$-th column. According to the preceding discussion, the absolute value of the determinate of $M'(D)$ equals $p$, which is a prime integer. Let $M'^{*}(D)$ be the adjoint of $M'(D)$ and $w_1, \cdots, w_{n-1}$ the columns of $M'^{*}(D)$, then it follows that
\begin{center}
$M'(D)w_i=0$ mod $p$.
\end{center}
Since $D$ is alternating, the sum of the rows of $M(D)$ is a zero vector. Hence we have
\begin{center}
$M(D)\begin{pmatrix}
    w_i\\
     0
\end{pmatrix}=0$ mod $p$.
\end{center}
In other words each $\begin{pmatrix}
    w_i\\
     0
\end{pmatrix}$ gives a proper coloring of $D$.

The key observation is that none of $w_i$ is a zero vector $($mod $p)$, or one can obtain a new diagram $D'$ from $D$ via some virtual moves such that $D'$ is a connected sum of two reduced alternating virtual knot diagrams. This contradicts the assumption that the determinate of $D$ is prime. In fact if $w_i=\overrightarrow{0}$ $($mod $p)$ for some $1\leq i\leq n-1$, it follows that
\begin{center}
$M(D)\begin{pmatrix}
    \frac{1}{p}w_i\\
     0
\end{pmatrix}=e_i-e_n$,
\end{center}
here $e_i$ denotes the column vector which has a 1 in the $i$th position as its only nonzero entry. Note that each entry of $\begin{pmatrix}
    \frac{1}{p}w_i\\
     0
\end{pmatrix}$ is an integer, hence it offers a pseudo $Z$-coloring of the diagram $D$. Here a \emph{pseudo coloring} \cite{Mat2009} means a way of coloring the arcs such that the zero vector on the right side of the equation $M(D)X=0$ is replaced by $e_i-e_n$. Next we show that if there is a pseudo coloring on $D$, then $D$ can be converted to the connected sum of two reduced alternating virtual knot diagrams. The method we use here mainly comes from the analysis given in \cite{Mat2009}, combining with some properties of virtual knots.

Consider the two-in two-out directed graph $G$ constructed in the proof of Proposition 2.3. Since the in-degree and out-degree are both 2 for each vertex, it follows that $G$ has a directed Euler circuit. We use $E$ to denote it. Since $D$ is alternating, each arc of $D$ is split into a pair of adjacent edges in $G$, hence the size $($the number of edges$)$ of $G$ is $2n$. If there exists a pseudo coloring $X$, i.e. $M(D)X=e_i-e_n$, let us consider the largest entry of $X$, say $h$. As before, let $v_j$ denote the vertex that corresponds to the crossing point $c_j$. By some simple calculation we conclude that
\begin{enumerate}
  \item if one ingoing edge of $v_j$ $(j\neq i, n)$ receives the color $h$, then the other three edges around $v_j$ must be colored by $h$;
  \item if one ingoing edge of $v_i$ receives the color $h$, then one outgoing edge is colored by $h-1$ and the other two edges are both colored by $h$;
  \item none ingoing edge of $v_n$ can receive the color $h$.
\end{enumerate}
Recall that $D$ is alternating. In the Euler circuit $E$ each vertex of $G$ will appear twice, according to the discussion above one of the $v_i$ in $E$ is the only vertex such that the color of the ingoing edge and that of the outgoing edge are $h$ and $h-1$ respectively. In other words if we consider a sub-path in $E$ such that all edges of it are colored by $h$, then the final vertex of it must be $v_i$. Therefore there is only one such sub-path in $E$, we use $H$ to denote it. Without loss of generality, we assume that $v_k$ is the initial vertex of $H$. Note that it is possible $k=n$.

Now let us go back to the diagram $D$. According to the discussion above, we find that one under-arc of $c_k$ $(c_i)$ receives the color $h$, but neither the other under-arc nor the over-arc is colored by $h$. Now we argue that if we chooses $c_k$ as the starting point and the under-arc with color $h$ as the first arc, when we travel along the diagram $D$ until meeting the under-crossing $c_i$ for the first time, all the arcs we visited are precisely all the arcs of $D$ with color $h$. First we show that all the arcs we meet are colored by $h$. In order to see this, let $c_m$ be the first crossing we meet such that the next under-arc is not colored by $h$. Obviously the color of the over-arc at $c_m$ can not be colored by $h$, hence $v_m$ must be an endpoint of the path $H$. It follows that $c_m=c_i$. On the other hand, if there are some other arcs with color $h$ that we did not visit, then the endpoints of the longest paths consisted of these arcs must be the endpoints of $H$. Since $H$ is the unique sub-path in $E$, all the arcs with color $h$ have been visited.

If the diagram $D$ contains no virtual crossing points, then all the arcs with color $h$ factors out of the diagram $D$. In fact if we still use $H$ to denote the arcs with color $h$, then $H$ constitutes a non-trivial summand of the diagram $D$. However if $D$ includes some virtual crossing points, this is not the case. Since the other portion $D-H$ may cross $H$ at some virtual crossings. For example Figure 5 gives a simple example of this.
\begin{center}
\includegraphics{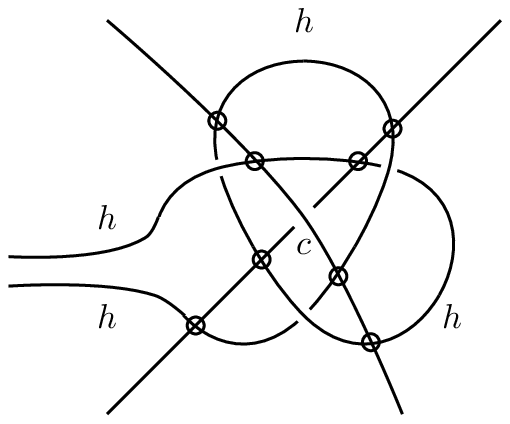} \centerline{\small Figure
5\quad}
\end{center}
In order to see $H$ still constitutes a summand of the diagram $D$, it suffices to move the classical and virtual crossings intersected by $D-H$ and itself out of the regions bounded by $H$. For instance, after moving the crossing $c$ in Figure 5 to the unbounded region the arcs with color $h$ constitutes a summand of the diagram. In general, for each region we can define the distance between this region and the unbounded region to be the least number of arcs that must be crossed by a path from a point in this region to a point far from the diagram. Every time we can move the crossings from a region to another adjacent region with smaller distance via some virtual moves. Repeating this process till we obtain a composite diagram. Notice that both the summands are non-trivial reduced alternating virtual knot diagram, by Proposition 2.3 this contradicts with the assumption that $p$ is prime.

Now we have proved that $w_i\neq \overrightarrow{0}$ mod $p$. According to \cite{Mat2009} we say a virtual knot diagram $D$ has a \emph{fundamental coloring} if for any two nontrivial colorings $X_1$ and $X_2$ there are integers $a, b$ such that $X_1=aX_2+bT$ mod $p$, here $T$ denotes the trivial coloring, i.e. the column vector with all entries 1. It was proved in \cite{Mat2009} that if a knot diagram has prime determinate then it has a fundamental coloring. The key point of this result is that $Z_p$ forms a finite field when $p$ is prime. With the similar idea we can prove that the diagram $D$ has a fundamental coloring. Obviously if $D$ admits one heterogeneous coloring then each nontrivial coloring of $D$ is heterogeneous. Hence it suffices to show $D$ admits a heterogeneous coloring. A beautiful observation in \cite{Mat2009} is that the coloring matrix $M(D_m)$ is the transpose of $M(D)$, here $D_m$ denotes the mirror image of the diagram $D$. Let $M'(D_m)$ be the submatrix of $M(D_m)$ formed by deleting the $n$-th row and $n$-th column and $M'^{*}(D_m)$ the adjoint of $M'(D_m)$. We use $u_1, \cdots, u_{n-1}$ to denote column vectors of $M'^{*}(D_m)$. Since $D_m$ is also a reduced alternating virtual knot diagram, then by the discussion above we have $u_i\neq \overrightarrow{0}$ mod $p$. Without loss of generality we assume that the first entry of $u_1\neq 0$ mod $p$. Recall that $\begin{pmatrix}
    u_1\\
     0
\end{pmatrix}$ gives a proper coloring of $D_m$, then this coloring distinguishes the arc $a_1$ from $a_n$. Since $D_m$ admits a fundamental coloring, it follows that all nontrivial colorings distinguish $a_1$ from $a_n$. As a corollary the first row vector of $M'^{*}(D_m)$ contains no 0 entries mod $p$, hence the first column vector of $M'^{*}(D)$ also contains no 0 entries. In other words there is a proper coloring of $D$ such that $a_n$ has different color from others. Repeating the argument above by replacing $a_n$ with $a_i$ $(1\leq i\leq n-1)$ we obtain that each nontrivial proper coloring of $D$ is heterogeneous. This finishes the proof.
\end{proof}

\end{document}